\documentclass[reqno]{amsart}

\usepackage[english]{babel}
\usepackage[all]{xy}
\usepackage{amssymb}
\usepackage{amsmath}
\usepackage{amsfonts}
\usepackage{url}
\usepackage{enumerate}
\usepackage{hyperref}
\usepackage{color}

\newtheorem{theorem}{Theorem}[section]
\newtheorem{prop}[theorem]{Proposition}
\newtheorem{lemma}[theorem]{Lemma}
\newtheorem{cor}[theorem]{Corollary}

\theoremstyle{definition}

\newtheorem{rem}[theorem]{Remark}

\newtheorem*{convention}{Conventions}

\newcommand{\IC}{\mathbb{C}}

\newcommand{\IG}{\mathbb{G}}

\newcommand{\IP}{\mathbb{P}}

\newcommand{\IZ}{\mathbb{Z}}

\newcommand{\cH}{\mathcal{H}}
\newcommand{\cI}{\mathcal{I}}

\newcommand{\cN}{\mathcal{N}}
\newcommand{\cO}{\mathcal{O}}
\newcommand{\cT}{\mathcal{T}}

\newcommand{\Ann}{\operatorname{Ann}}

\newcommand{\Grass}{\operatorname{Grass}}
\newcommand{\Hilb}{\operatorname{Hilb}}

\newcommand{\im}{\operatorname{im}}

\newcommand{\Pic}{\operatorname{Pic}}
\newcommand{\rk}{\operatorname{rk}}
\newcommand{\Sec}{\operatorname{Sec}}
\newcommand{\Sing}{\operatorname{Sing}}
\newcommand{\Tan}{\operatorname{Tan}}

\newcommand{\chitop}{\chi_{\text{top}}}
\newcommand{\res}{\text{res}}
\newcommand{\st}{\, \middle| \,}

\newcommand{\pa}[1]{\left( #1 \right)}
\newcommand{\set}[1]{\left\{ #1 \right\}}

\title[Degree of Tangent and Secant Variety to Surfaces]{The Degree of the Tangent and Secant Variety to a Projective Surface}

\begin{document}

\author{Andrea Cattaneo}
\address{Andrea Cattaneo, Institut Camille Jordan UMR 5208, Universit\'e Claude Bernard Lyon 1, 69622 Villeurbanne Cedex, France}
\email{cattaneo@math.univ-lyon1.fr}

\thanks{The author is supported by the LABEX MILYON (ANR-10-LABX-0070) of Universit\'e de Lyon, within the program ``Investissements d'Avenir'' (ANR-11-IDEX- 0007) operated by the French National Research Agency (ANR) and is member of GNSAGA of INdAM. The main parts of this paper were written while the author was granted a research fellowship by Universit\`a degli Studi dell'Insubria in Como. He wants to gratefully acknowledge all the people who helped him during the preparation of this paper: in particular S.~Boissi\`ere and A.~Sarti who suggested the problem from which the present paper originated, and C.~Ciliberto, F.~Flamini and A.~Lanteri for useful comments on the subject and suggestions during the writing. He also wants to thank the referee for his/her suggestions, which led to a substantial improvement of the first version of this paper.}


\begin{abstract}
In this paper we present a way of computing the degree of the secant (resp., tangent) variety of a smooth projective surface, under the assumption that the divisor giving the embedding in the projective space is $3$-very ample. This method exploits the link between these varieties and the Hilbert scheme $0$-dimensional subschemes of length $2$ of the surface.
\end{abstract}



\maketitle

\section*{Introduction}

In this paper we study the link between the secant variety of a smooth projective complex surface $S$ and the Hilbert scheme of $0$-dimensional subschemes of length $2$ of $S$. In particular, when the embedding of $S$ in the projective space $\IP^n$ is given by a $2$-very ample divisor, then (cf. \cite{cat-goe}) we can identify the Hilbert scheme $\Hilb^2 S$ with the subvariety of the Grassmannian $\IG(1, n)$ parametrising all the secant lines to $S$. We can then write (the class of) $\Hilb^2 S$ as a linear combination of Schubert cycles of $\IG(1, n)$, and we expect the coefficients of this linear combination to have some geometric meaning. In particular, they should reflect how the surface is embedded in $\IP^n$. As a consequence, denoting by $h$ the very ample divisor which embeds $S$ in $\IP^n$, we expect these coefficients to depend both on some intrinsic invariant of the surface, such as $K_S^2$, and on some properties of the embedding, such as the degree $h^2$ of the embedded surface. Our purpose is then to explicitly compute these coefficients and give also their enumerative interpretation: it turns out that one of them is exactly the degree of the secant variety of the embedded surface, whose value is computed in Theorem \ref{thm: degree secant variety}.

Under the assumption that the divisor $h$ is $3$-very ample, we can study in an analogous way the subvariety of the Grassmannian $\IG(2, n)$ parametrising the tangent planes to $S$ (i.e., the image of the Gauss map of $S$). As a result, we give a formula for the degree of the tangent variety to $S$ in Theorem \ref{thm: degree tangent variety}.

We shall say that these results can be obtained also in other more classical ways, for example as an application of the double point formula (see, e.g., \cite[Theorem 9.3]{intersection-theory} or \cite[$\S$0]{catanese}), but in our exposition we want to emphasize and exploit the link between the enumerative properties of the surface $S$ and the geometry of its Hilbert scheme $\Hilb^2 S$. Nevertheless, our formulae are quite simple and explicit and we can easily apply them to some examples of special interest: the case of a generic projective $K3$ surface and the case of the Veronese surfaces in $\IP^n$.

Our method, in principle, can also be used to study higher secant varieties to surfaces embedded by $k$-very ample divisors, but there are some difficulties. First of all, the growth of the dimension of the higher secant varieties forces one to work with higher-dimensional Schubert cycles in the Grassmannian $\IG(k - 1, n)$, which leads to a corresponding growth in the complexity of the computations. Moreover, in this paper we exploit a fact which is peculiar of the Hilbert scheme $\Hilb^2 S$, namely that it can be defined as a quotient of the blow up of $S \times S$ in the diagonal. This is not the case for $\Hilb^k S$ with $k > 2$ (cf. Remark \ref{rem: hilbert scheme construction}), and as a consequence in order to compute the degree of the higher secant varieties one needs to be more careful with the intersection theoretic arguments we propose here.

The structure of the paper is as follows. In Section \ref{sect: definitions} we recall the construction of the secant variety and of the Hilbert scheme of a projective surface, explaining in the last part the relation between them. In Section \ref{sect: secant variety general} we give some general results on the secant variety to a projective surface: in particular, we determine the class of the variety parametrising the secant lines to $S$ in the corresponding Grassmannian; we give an explicit enumerative interpretation of all the coefficients involved and compute all of them except the one corresponding to the degree of the secant variety of the surface. In Section \ref{sect: tangent variety general} we focus on the subvariety parametrising the tangent planes to the surface: writing its class in the cohomology ring of the Grassmannian as a combination of Schubert cycles, we can provide an enumerative interpretation of the coefficients which appear. In Section \ref{sect: intersection in hilb2} we recall some facts on the intersection theory in the Hilbert scheme, and determine the intersection number we will need in Section \ref{sect: deg sec s, deg tan s} to compute the degree of the secant and tangent variety. In Section \ref{sect: deg sec s, deg tan s} we exploit a linear relation and the intersection numbers computed in the previous Sections to prove our main results: in Theorem \ref{thm: degree secant variety} we compute the degree of the secant variety and in Theorem \ref{thm: degree tangent variety} the one of the tangent variety. Finally, in Section \ref{sect: examples} we specialise our result in two concrete situations, the one of $K3$ surfaces and of the image of $\IP^2$ under the $m$-th Veronese embedding.

\begin{convention}
Throughout we will work over the field $\IC$ of complex numbers. Moreover, we always implicitly assume that the embedded surfaces are non-de\-fec\-tive, meaning that their secant and tangent variety have the expected dimension, i.e., $\dim \Sec S = 5$ and $\dim \Tan S = 4$ (this will be a harmless constraint, see the discussion in Section \ref{sect: veronese}). As a consequence, from Section \ref{sect: secant variety general} we will always consider surfaces $S \subseteq \IP^n$ with $n \geq 5$.
\end{convention}


\section{The secant variety and the Hilbert scheme}\label{sect: definitions}

In this Section we briefly recall the definition and the construction of the secant variety $\Sec S$ to a smooth projective surface $S$ as well as of the Hilbert scheme $\Hilb^2 S$ parametrising the $0$-dimensional subschemes of length $2$ of $S$. Our main references are \cite[$\S$1(a)]{dale} and \cite[$\S$6]{bea} respectively.

We introduce here the notation we use throughout the paper for the Grassmannians: $\Grass(k, n)$ denotes the Grassmannian parametrising the $k$-dimensional subspaces of a complex $n$-dimensional vector space, while $\IG(k, n)$ denotes the Grassmannian parametrising the $k$-dimensional subspaces of a complex $n$-dimensional projective space. Hence $\IG(k, n) = \Grass(k + 1, n + 1)$.

\subsection{The secant variety}

Let $S \subseteq \IP^n$ be a smooth projective surface, which is not contained in any hyperplane. Consider the map
\[\begin{array}{rccl}
f: & (S \times S) \smallsetminus \Delta_S & \longrightarrow & \IG(1, n)\\
 & (P, Q) & \longmapsto & \text{line } \langle P, Q \rangle,
\end{array}\]
where $\Delta_S$ is the diagonal of $S \times S$. Let $\Gamma(S)$ be the closure of the graph of $f$ in $\IP^n \times \IP^n \times \IG(1, n)$, and $\Sigma(S)$ be the image of $\Gamma(S)$ in $\IG(1, n)$ under the projection on the last factor. Then $\Sigma(S)$ is the subset of the Grassmannian which parametrises the lines which are secant to $S$.

In order to define the secant variety, we consider the incidence variety
\[I = \set{(x, l) \in \IP^n \times \IG(1, n) \st x \in l},\]
and restrict it to the set of secant lines: we let $\Sigma B(S)$ be the inverse image in $I$ of $\Sigma(S)$ under the projection on the second factor. The \emph{secant variety}, $\Sec S$, of $S$ is the image of $\Sigma B(S)$ in $\IP^n$.

We have then the following situation (we denote by $pr$ the projections, and use the subscripts to indicate the factors):
{\small
\begin{equation}\label{eq: first diagram}
\xymatrix{\IP^n \times \IP^n \times \IG(1, n) \supseteq & \Gamma(S) \ar[dl]_{pr_{12}} \ar[dr]^{pr_3} & & \Sigma B(S) \ar[dl]_{pr_2} \ar[dr]^{pr_1} & \subseteq I \subseteq \IP^n \times \IG(1, n)\\
S \times S & & \Sigma(S) & & \Sec S}
\end{equation}}

We recall that
\begin{enumerate}
\item $pr_{12}$ is the blow up of $S \times S$ along $\Delta_S$;
\item $pr_3$ is generically finite, and if the generic secant line cuts $S$ in $m$ distinct points, then it is $m(m - 1) : 1$. In particular:
\begin{enumerate}
\item if $n = 3$, and $S$ is a surface of degree $d$, then $m = d$,
\item if $n > 3$, then by \cite[Theorem 1.8]{dale} or \cite[Corollary 2.7, Corollary 2.8]{holme-roberts} we have $m = 2$, and so $pr_3$ is $2:1$;
\end{enumerate}
\item $pr_2$ is a $\IP^1$-bundle;
\item the fibre $pr_1^{-1}(x)$ represents all the secants passing through $x \in \Sec S$.
\end{enumerate}

\subsection{The Hilbert scheme}

Let $S$ be a projective surface. The (second) symmetric product $S^{(2)}$ of $S$ is the quotient of $S \times S$ by the involution exchanging the two factors: it is the variety representing the effective $0$-dimensional cycles on $S$, and it is singular along the image of the diagonal. Let $\varepsilon: \Hilb^2 S \longrightarrow S^{(2)}$ be the blow up of the singular locus. We obtain then a smooth variety, whose points parametrize the $0$-dimensional subschemes of length $2$ of $S$. The morphism $\varepsilon$ is called the Hilbert--Chow morphism.

Another way to define the Hilbert scheme $\Hilb^2 S$ is to blow up $S \times S$ along the fixed locus of the involution, i.e., the diagonal $\Delta_S$, and then take the quotient of the blown up variety by the induced involution. This leads to the commutative square
\begin{equation}\label{eq: second diagram}
\xymatrix{\widetilde{S \times S} \ar[r]^\eta \ar[d]_\rho & S \times S \ar[d]^\pi\\
\Hilb^2 S \ar[r]^\varepsilon & S^{(2)},}
\end{equation}
and we recall that the action induced on $\widetilde{S \times S}$ by the one on $S \times S$ acts as the identity on the exceptional divisor.

\begin{rem}\label{rem: hilbert scheme construction}
We can define $\Hilb^k S$ as the variety parametrising the $0$-dimensional subschemes of $S$ of length $k$. The first of the two constructions we recalled for $\Hilb^2 S$ generalises to the Hilbert schemes $\Hilb^k S$. However, the second one does not generalise: call $\Delta_{ij}$ the subset of $S^k$ of the $k$-uples $(x_1, \ldots, x_k)$ with $x_i = x_j$, and let $\Delta = \pi \pa{\bigcup_{i, j} \Delta_{ij}}$; then $\Delta$ is the singular locus of $S^{(k)}$ and $\varepsilon$ coincides with the blow up of $\Delta$ only on an open part of $S^{(k)}$ (the image in $S^{(k)}$ of the set of $k$-uples with at most two equal entries), whose complement has codimension at least $3$ (cf. \cite[$\S$6, p. 766]{bea}). Because of this, the methods we will describe may be hard to generalise to higher dimensional secant varieties.
\end{rem}

\subsection{\texorpdfstring{$k$}{k}-very ampleness}
We now recall the concept of $k$-very ampleness (cf. \cite[$\S$2]{beltrametti-sommese} and \cite{cat-goe}). Let $S$ be a surface, and $h$ be a divisor on it. Fix a $0$-dimensional subscheme $Z$ of $S$ of length $k + 1$, defined by the ideal sheaf $\cI_Z$, and consider the exact sequence
\[0 \longrightarrow \cI_Z \otimes \cO_S(h) \longrightarrow \cO_S(h) \longrightarrow \cO_Z \otimes \cO_S(h) \longrightarrow 0.\]
This sequence induces the long cohomology sequence
\begin{equation}\label{eq: ses defining res}
0 \longrightarrow H^0(S, \cI_Z \otimes \cO_S(h)) \longrightarrow H^0(S, \cO_S(h)) \xrightarrow{\res_Z} H^0(Z, \cO_Z \otimes \cO_S(h)) \longrightarrow \ldots,
\end{equation}
and we say (cf. \cite[Definition 0.1(iii)]{cat-goe}) that $h$ is \emph{$k$-very ample} if the restriction map $\res_Z$ in \eqref{eq: ses defining res} is onto for every $0$-dimensional subscheme $Z$ of length at most $k + 1$.

It is immediate to see that $0$-very ampleness is equivalent to global generation, and that $1$-very ampleness is equivalent to very ampleness. Moreover, any $(k - 1)$-very ample divisor $h$ induces a map
\begin{equation}\label{eq: hilb to grass map}
\varphi_{k - 1}: \Hilb^k S \longrightarrow \Grass(k, H^0(S, \cO_S(h))^*),
\end{equation}
associating to any $0$-dimensional subscheme of length $k$ of $S$ the point representing the $k$-dimensional subspace $H^0(Z, \cO_Z \otimes \cO_S(h))^*$ in $H^0(S, \cO_S(h))^*$. The answer to the question whether this map is an embedding is given in the following Theorem.

\begin{theorem}[{\cite[Main Theorem]{cat-goe}}]\label{thm: cat-goe}
The map $\varphi_{k - 1}$ defined in \eqref{eq: hilb to grass map} is an embedding if and only if $h$ is $k$-very ample.
\end{theorem}

\subsection{The link}

We begin to see the link between the Hilbert scheme and the secant variety since the diagrams \eqref{eq: first diagram} and \eqref{eq: second diagram} overlap:
\[\xymatrix{ & \Gamma(S) = \widetilde{S \times S} \ar[rr]^(0.6){\eta = pr_{12}} \ar[dl]_{\rho} \ar[dr]^{pr_3} & & S \times S\\
\Hilb^2 S & & \Sigma(S) & \subseteq \IG(1, n).}\]

As observed, the action on $S \times S$ is the one exchanging the two factors, and fixes the diagonal $\Delta_S$. The induced action on $\widetilde{S \times S}$ coincides with this one outside the exceptional divisor, and fixes it pointwise. This implies that the morphism $pr_3: \widetilde{S \times S} \longrightarrow \Sigma(S)$ is constant on the orbits of the action, and so we have a morphism $\varphi$ making the following diagram commute:
\begin{equation}\label{eq: hilb to sigma s}
\xymatrix{ & \Gamma(S) = \widetilde{S \times S} \ar[dl]_{\rho} \ar[dr]^{pr_3} & \\
\Hilb^2 S \ar[rr]^\varphi & & \Sigma(S).}
\end{equation}

We now want to compare this map with the map $\varphi_1$ defined in \eqref{eq: hilb to grass map}.

\begin{lemma}\label{lemma: geometric description}
Let $h$ be a very ample divisor on the surface $S$. Then the maps $\varphi_1$ of \eqref{eq: hilb to grass map} and $\varphi$ of \eqref{eq: hilb to sigma s} coincide.
\end{lemma}
\begin{proof}
Since both $\varphi_1$ and $\varphi$ are morphisms, it suffices to show that they agree set-theoretically on a dense open subset of $\Hilb^2 S$. We will consider then the open subset obtained as the complement to the exceptional divisor of the Hilbert--Chow morphism $\varepsilon$. This open set parametrises the length $2$ subschemes of $S$ supported on two distinct points.

Use $h$ to embed $S$ in $\IP^n = \IP(H^0(S, \cO_S(h))^*)$. Let $Z$ be a $0$-dimensional subscheme of $S$ of length $2$, defined by the ideal sheaf $\cI_Z$, whose support consists of the distinct points $P$ and $Q$. Consider then $(P, Q) \in \widetilde{S \times S}$ a lift of $Z$, and observe that via the identification $\widetilde{S \times S} = \Gamma(S)$ this point corresponds to $(P, Q, \langle P, Q \rangle)$. As a consequence, $pr_{13}(P, Q, \langle P, Q \rangle) = \langle P, Q \rangle$, and since this expression is clearly symmetric in $P$ and $Q$ we deduce that
\[\varphi(Z) = \text{line through } Z \text{ in } \IP(H^0(S, \cO_S(h))^*)\]
on this open subset.

We will now show that $\varphi_1(Z)$ admits the same description on the open subset we are considering. Since $\cO_S(h)$ is a very ample line bundle on $S$, we have the short exact sequence
\[0 \longrightarrow H^0(S, \cI_Z \otimes \cO_S(h)) \longrightarrow H^0(S, \cO_S(h)) \longrightarrow H^0(Z, \cO_Z \otimes \cO_S(h)) \longrightarrow 0,\]
whose dual
\[0 \longrightarrow H^0(Z, \cO_Z \otimes \cO_S(h))^* \longrightarrow H^0(S, \cO_S(h))^* \longrightarrow H^0(S, \cI_Z \otimes \cO_S(h))^* \longrightarrow 0\]
shows that $H^0(Z, \cO_Z \otimes \cO_S(h))^* = \Ann H^0(S, \cI_Z \otimes \cO_S(h))$. But then
\[\begin{array}{rl}
\varphi_1(Z) = & \text{line } \IP(H^0(Z, \cO_Z \otimes \cO_S(h))^*) \text{ in } \IP(H^0(S, \cO_S(h))^*) =\\
= & \text{line } \IP(\Ann H^0(S, \cI_Z \otimes \cO_S(h))) \text{ in } \IP(H^0(S, \cO_S(h))^*) =\\
= & \text{line through } Z \text{ in } \IP(H^0(S, \cO_S(h))^*).
\end{array}\]
\end{proof}

\begin{cor}
Let $h$ be a very ample divisor on the surface $S$. Then the morphism $\varphi$ of \eqref{eq: hilb to sigma s} is an embedding if and only of $h$ is $2$-very ample.
\end{cor}
\begin{proof}
In fact this is true for $\varphi_1$ of \eqref{eq: hilb to grass map} by Theorem \ref{thm: cat-goe}.
\end{proof}

\begin{rem}
Let $S$ be a surface embedded in $\IP^n$ by means of the very ample divisor $h$. The map $\varphi_{k - 1}$ defined in \eqref{eq: hilb to grass map} associates to any $0$-dimensional subscheme $Z$ of length $k$ of $S$ the point representing the linear subspace of $\IP^n$ spanned by $Z$, i.e., the unique $(k - 1)$-dimensional linear subspace of $\IP^n$ containing $Z$.
\end{rem}

Thanks to this geometric description it is now easy to see that if the embedding of $S$ in $\IP^n$ is given by a $2$-very ample line bundle, then (the image of) $S$ contains no lines.

\begin{prop}[{\cite[(0.5.1)]{bs-kspannedness}}]
Let $S$ be a surface and $h$ be a very ample divisor on it. If there exists a divisor $l$ such that $l \simeq \IP^1$ and $\deg \cO_S(h)_{|_l} = 1$, then $h$ is not $2$-very ample.
\end{prop}
\begin{proof}
This Proposition is a classical result on $k$-very ampleness (see, e.g., \cite{beltrametti-sommese} or \cite[(0.5.1)]{bs-kspannedness}), it also follows directly from the geometric description of the map $\varphi$ given in the proof of Lemma \ref{lemma: geometric description}, as we show now.

Assume by contradiction that $h$ is $2$-very ample. In particular, it is very ample and so embeds $S$ in $\IP^n = \IP(H^0(S, \cO_S(h))^*)$ in such a way that $l$ is a line contained in $S$. From Lemma \ref{lemma: geometric description} we deduce that the map $\varphi$ is constant on $\Hilb^2 l \subseteq \Hilb^2 S$, with value the point of $\IG(1, n)$ corresponding to the line $l$, thus contradicting the fact that $\varphi$ is an embedding.
\end{proof}

\begin{rem}\label{rem: geom descr}
Assume that $S$ is embedded in $\IP^n$ by means of a $k$-very ample divisor $h$, with $k \geq 2$. This embedding has the property that for any $0$-dimensional subscheme $Z$ of $S$ of length $k$, the linear subspace spanned by $Z$ intersects $S$ \emph{exactly} in $Z$ (a priori, it could have cut on $S$ a subscheme of higher length containing $Z$). In particular:
\begin{enumerate}
\item If $k \geq 2$, then for any $0$-dimensional subscheme $Z$ of $S$ of length $2$ the secant line spanned by $Z$ has no further intersections with $S$. As a consequence, a tangent line meets $S$ only in the tangency point, and so we deduce that for any $P \in S$ we have $T_P S \cap S = \{ P \}$.
\item If $k \geq 3$, then for any pair of distinct $0$-dimensional subschemes $Z$, $Z'$ of $S$ of length $2$, the secant lines spanned by $Z$ and $Z'$ are either disjoint or they meet in a point of $S$. In fact, if they intersect away from $S$, then these two lines span a plane which contains a length $4$ subscheme of $S$, which contradicts the $3$-very ampleness of $h$. As a consequence, each point in $\Sec S \smallsetminus S$ belongs to one and exactly one secant line. In particular, any pair of tangent planes are disjoint, i.e., for $P, Q \in S$ we have that $T_P S \cap T_Q S \neq \varnothing$ if and only if $P = Q$.
\end{enumerate}
\end{rem}

\section{The secant variety in the Grassmannian}\label{sect: secant variety general}

In this Section we want to present a strategy to determine the degree of the secant variety of a smooth surface.

Let $S$ be a smooth projective surface, embedded in $\IP^n$ by means of the very ample divisor $h$: call $d = h^2$ the degree of $S$ in $\IP^n$. In this Section we assume that $n \geq 5$ and that $h$ is $3$-very ample: as a consequence we have $\Hilb^2 S \simeq \Sigma(S) \subseteq \IG(1, n)$. We want to determine the class $[\Sigma(S)]$ of $\Sigma(S)$ in the cohomology ring of $\IG(1, n)$, so we begin this Section describing the Schubert cycles we are interested in the Grassmannian.

As a matter of notation, we will denote by $\Lambda_k$ a generic $k$-dimensional linear subspace of $\IP^n$, and given a subvariety $W$ in a manifold $V$ we denote by $[W]$ the class of $W$ in the cohomology ring of $V$.

Throughout this Section, we refer to \cite[$\S$1.5]{gh} both for the notation for Schubert cycles ($\sigma_{a_1, \ldots, a_k}$) and the intersection theoretic properties that we use. In particular, we recall that the Schubert cycle $\sigma_{a_1, \ldots, a_k}$ in $\Grass(k, n)$ parametrises all the $\Lambda_k \subseteq V_n$ such that
\[\dim (\Lambda_k \cap L_{n - k + 1 - a_i}) \geq i \qquad \text{for all } i,\]
where $V_n$ is a complex vector space of dimension $n$ and $0 = L_0 \subseteq L_1 \subseteq \ldots \subseteq L_{n - 1} \subseteq L_n = V_n$ is a complete flag in $V_n$.

\subsection{The variety of secant lines in the Grassmannian}
The dimension of $\Grass(2, n + 1)$ is $2(n - 1)$, and the Pl\"ucker map embeds it in $\IP^{N - 1}$ with $N = \binom{n + 1}{2} = \frac{n(n + 1)}{2}$. The Pl\"ucker embedding is induced by the linear system associated to the Schubert cycle $\sigma_{1, 0}$, which represents (i.e., its points are in bijection with) all the lines in $\IP^n$ which meet a fixed $\Lambda_{n - 2}$. This means that
\[P = \varphi_{|\sigma_{1, 0}|}: \IG(1, n) \longrightarrow \IP^{\frac{(n + 2)(n - 1)}{2}}, \qquad P^*(\cH) = \sigma_{1, 0},\]
where $P$ is the Pl\"ucker embedding and $\cH$ is a hyperplane in $\IP^n$.

There are three Schubert cycles of codimension $4$ in $\IG(1, n)$:
\begin{enumerate}
\item $\sigma_{4, 0}$, which represents all the lines in $\IP^n$ meeting a fixed $\Lambda_{n - 5}$;
\item $\sigma_{3, 1}$, which represents all the lines in $\IP^n$ contained in a fixed $\Lambda_{n - 1}$ and meeting a fixed $\Lambda_{n - 4} \subseteq \Lambda_{n - 1}$;
\item $\sigma_{2, 2}$, which represents all the lines in $\IP^n$ contained in a fixed $\Lambda_{n - 2}$.
\end{enumerate}
There are three Schubert cycles of dimension $4$, i.e., of codimension $2(n - 3)$, in $\IG(1, n)$:
\begin{enumerate}
\item $\sigma_{n - 1, n - 5}$, which represents all the lines in $\IP^n$ contained in a fixed $\Lambda_5$ through a fixed point $P \in \Lambda_5$;
\item $\sigma_{n - 2, n - 4}$, which represents all the lines in $\IP^n$ contained in a fixed $\Lambda_4$ and meeting a fixed line $\Lambda_1 \subseteq \Lambda_4$;
\item $\sigma_{n - 3, n - 3}$, which represents all the lines in $\IP^n$ contained in a fixed $\Lambda_3$.
\end{enumerate}

\begin{rem}
We have then that $\sigma_{n - 3, n - 3} \simeq \IG(1, 3)$, that $\sigma_{n - 2, n - 4}$ is isomorphic to the Schubert cycle $\sigma_{2, 0} \subseteq \IG(1, 4)$ and finally that $\sigma_{n - 1, n - 5} \simeq \IP^4$.
\end{rem}

These Schubert cycles intersect according to Table \ref{tab: intersection g(1, n)}.

\begin{table}[h]
\begin{tabular}{||c||c|c|c||}
\hline
\hline
$\cdot$ & $\sigma_{4, 0}$ & $\sigma_{3, 1}$ & $\sigma_{2, 2}$\\
\hline
\hline
$\sigma_{n - 1, n - 5}$ & $1$ & $0$ & $0$\\
\hline
$\sigma_{n - 2, n - 4}$ & $0$ & $1$ & $0$\\
\hline
$\sigma_{n - 3, n - 3}$ & $0$ & $0$ & $1$\\
\hline
\hline
\end{tabular}
\caption{Intersection table of the Schubert cycles of dimension and of codimension $4$ in $\IG(1, n)$.}
\label{tab: intersection g(1, n)}
\end{table}

Since $h$ is at least $2$-very ample on $S$, we can write
\begin{equation}\label{eq: decomp in g(1, n)}
[\Sigma(S)] = \alpha \sigma_{n - 1, n - 5} + \beta \sigma_{n - 2, n - 4} + \gamma \sigma_{n - 3, n - 3},
\end{equation}
and we want to compute the coefficients $\alpha$, $\beta$ and $\gamma$. Thanks to table \ref{tab: intersection g(1, n)}, we have that
\[\alpha = [\Sigma(S)] \cdot \sigma_{4, 0}, \qquad \beta = [\Sigma(S)] \cdot \sigma_{3, 1}, \qquad \gamma = [\Sigma(S)] \cdot \sigma_{2, 2}.\]
By the $2$-very-ampleness of $h$ (and Remark \ref{rem: geom descr}), a line cutting $S$ in $2$ points can not have further intersections with $S$, and so we can characterize $\Sigma(S) \subseteq \IG(1, n)$ as
\begin{equation}\label{eq: sigma s in grass}
\Sigma(S) = \set{l \in \IG(1, n) \st \operatorname{length} (l \cap S) = 2}.
\end{equation}

\subsection{Enumerative meaning of the coefficients}

In this Subsection we give an enumerative meaning of the coefficients $\alpha$, $\beta$ and $\gamma$ introduced in \eqref{eq: decomp in g(1, n)}. We compute two of them in this Section in terms of numerical properties of the surface, and compute the last one in Section \ref{sect: deg sec s, deg tan s}.

We start by determining $\gamma$. Since $\gamma = [\Sigma(S)] \cdot \sigma_{2, 2}$, by \eqref{eq: sigma s in grass} it coincides with the number of lines in $\IP^n$ which are contained in a given $\Lambda_{n - 2}$ and meet $S$ in $2$ points. Now, a generic $\Lambda_{n - 2}$ cuts $d = h^2$ distinct points on $S$, and so we have at most $\binom{d}{2}$ lines. Since $h$ is at least $2$-very ample, no three of those points lie on the same line (cf. Remark \ref{rem: geom descr}), hence this gives a proof of the following Lemma.

\begin{lemma}\label{lemma: gamma}
The value of the coefficient $\gamma$ in \eqref{eq: decomp in g(1, n)} is
\[\gamma= \frac{1}{2} h^2 (h^2 - 1).\]
\end{lemma}

Now we determine $\beta$.

\begin{rem}\label{rem: jacobian criterion}
In this Remark we explain an easy consequence of the Jacobian criterion for smoothness, which will be used in the proof of Lemma \ref{lemma: beta} and Lemma \ref{lemma: sing curve in pencil}. Let $S \subseteq \IP^n$ be a smooth variety and let $H$ be a smooth hypersurface not containing $S$. Let $P \in S \cap H$: choosing a set of (local) equations $f_1, \ldots, f_m$ for $S$ and an equation $f$ for $H$ which are centred in $P$, we have that $P$ is singular for $S \cap H$ if and only if
\[\rk \left( \begin{array}{ccc}
\frac{\partial f}{\partial x_0}_{|_P} & \cdots & \frac{\partial f}{\partial x_n}_{|_P}\\
\hline
\frac{\partial f_1}{\partial x_0}_{|_P} & \cdots & \frac{\partial f_1}{\partial x_n}_{|_P}\\
\vdots & \ddots & \vdots\\
\frac{\partial f_m}{\partial x_0}_{|_P} & \cdots & \frac{\partial f_m}{\partial x_n}_{|_P}
\end{array} \right) < n - \dim(S \cap H) = n - \dim S + 1.\]
As the submatrix $\left( \frac{\partial f_i}{\partial x_j}_{|_P} \right)_{\renewcommand\arraystretch{0.5} \begin{array}{l}
\scriptscriptstyle 1 \leq i \leq m\\
\scriptscriptstyle 0 \leq j \leq n
\end{array}}$ has rank $n - \dim S$ since $P$ is a smooth point, hence the condition that $P$ is singular for $S \cap H$ becomes equivalent to $T_P S \subseteq T_P H$. In the following, we will use this observation when $S$ is a smooth non-degenerate surface and $H$ a is hyperplane.
\end{rem}

Since $\beta = [\Sigma(S)] \cdot \sigma_{3, 1}$, by \eqref{eq: sigma s in grass} it coincides with the number of lines in $\IP^n$ which are contained in a given $\Lambda_{n - 1}$, meet a given $\Lambda_{n - 4} \subseteq \Lambda_{n - 1}$ and cut $S$ in $2$ points.

\begin{lemma}\label{lemma: beta}
Let $S \subseteq \IP^n$ be a surface embedded by a $3$-very ample divisor $h$. Let $C = S \cap \Lambda_{n - 1}$ be an irreducible smooth hyperplane section of $S$. Then the coefficient $\beta$ in \eqref{eq: decomp in g(1, n)} coincides with the degree of the secant variety of $C$ in $\Lambda_{n - 1}$, and its value is
\[\beta = \frac{1}{2}(h^2 (h^2 - 4) - h \cdot K_S).\]
\end{lemma}
\begin{proof}
For a generic choice of $\Lambda_{n - 1} \subseteq \IP^n$, the curve $C = S \cap \Lambda_{n - 1}$ is an irreducible and smooth curve of genus
\[g(C) = 1 + \frac{1}{2} h \cdot (K_S + h),\]
where $K_S$ is a canonical divisor on $S$, and its degree in $\Lambda_{n - 1}$ is $d = h^2$.

A line $l$ contained in $\Lambda_{n - 1}$ and meeting $S$ in a length $2$ subscheme can be of the following types:
\begin{enumerate}
\item $l$ is a secant of $C$;
\item $l$ is a tangent of $C$;
\item $l$ is tangent to $S$ in a point $P$ of $C$, but $l$ is not the tangent of $C$ at $P$.
\end{enumerate}
Assume we are in the third case: then $\Lambda_{n - 1}$ contains $l$ and the tangent line $T_P C$ to $C$ at $P$, since $C$ is smooth. Then $\Lambda_{n - 1}$ contains the linear subspace generated by these two lines in $\Lambda_{n - 1}$, i.e., $T_P S$, and by Remark \ref{rem: jacobian criterion} this implies that $C$ is singular at $P$, which is a contradiction. Hence a line $l \subseteq \Lambda_{n - 1}$ meeting $S$ in $2$ points is of the first or second kind. Then $l$ is contained in the secant variety of $C$ in $\Lambda_{n - 1}$.

The secant variety $\Sec C$ is a threefold contained in $\Lambda_{n - 1}$, and so $\beta$ can as well be computed as the degree of $\Sec C$ in $\Lambda_{n - 1}$: a generic linear subspace of codimension $3$ meets this secant variety in a point which is not on $C$, and such a point uniquely determines a secant line by the $3$-very ampleness of $h$ by Remark \ref{rem: geom descr}. By \cite[Theorem 4.3]{dale} or \cite[Theorem 3.5]{ac-man-sch}, the degree of the secant variety to a smooth curve of genus $g$ and degree $d$ is
\[\deg \Sec C = \binom{d - 1}{2} - g,\]
which in our case says that
\[\beta = \frac{1}{2}(h^2 (h^2 - 4) - h \cdot K_S).\]
\end{proof}

The proof of Lemma \ref{lemma: beta} can be adapted to show that the coefficient $\alpha$ in \eqref{eq: decomp in g(1, n)} is the degree of the secant variety to $S$ in $\IP^n$.

\begin{lemma}\label{lemma: alpha}
Let $S \subseteq \IP^n$ be a surface embedded by a $3$-very ample divisor $h$. The coefficient $\alpha$ in \eqref{eq: decomp in g(1, n)} coincides with the degree of the secant variety of $S$.
\end{lemma}
\begin{proof}
Since $\alpha = [\Sigma(S)] \cdot \sigma_{4, 0}$, by \eqref{eq: sigma s in grass} it coincides with the number of lines in $\IP^n$ which meet a fixed $\Lambda_{n - 5}$ and cut $S$ in $2$ points. The secant variety of $S$ is $5$-dimensional, so a generic linear subspace of codimension $5$ cuts $\Sec S$ in $\deg \Sec S$ distinct points which do not lie on $S$. By the $3$-very ampleness of $h$ and Remark \ref{rem: geom descr}, each such point determines uniquely a secant line to $S$ and the Lemma follows.
\end{proof}

The degree of $\Sigma(S)$ in $\IP^{\frac{(n + 2)(n - 1)}{2}}$ under the Pl\"ucker embedding $P$ can easily be computed in terms of the coefficients $\alpha$, $\beta$ and $\gamma$. As the restriction of the hyperplane class to the Grassmannian is the Schubert cycle $\sigma_{1, 0}$, this degree coincides with
\[\deg (\sigma_{1, 0}^4 \cdot [\Sigma(S)]).\]
Using Pieri formula, we can compute that $\sigma_{1, 0}^4 = \sigma_{4, 0} + 3 \sigma_{3, 1} + 2 \sigma_{2, 2}$, and
\begin{equation}\label{eq: degree embedding}
\deg (\sigma_{1, 0}^4 \cdot [\Sigma(S)]) = \alpha + 3 \beta + 2 \gamma.
\end{equation}

To determine the value of $\alpha$ is then the same as to determine the degree of $\sigma_{1, 0}^4 \cdot [\Sigma(S)]$.

\section{The tangent variety}\label{sect: tangent variety general}

In the same spirit we defined the secant variety to a surface $S$ embedded in $\IP^n$, we can define the \emph{tangent variety} to $S$ in the following way. We consider the Gauss map
\[\begin{array}{rccc}
T: & S & \longrightarrow & \IG(2, n)\\
 & P & \longmapsto & T_P S,
\end{array}\]
and the incidence relation
\[I' = \set{(x, \pi) \in \IP^n \times \IG(2, n) \st x \in \pi} \subseteq \IP^n \times \IG(2, n).\]
We denote $\cT(S) = \im T \subseteq \IG(2, n)$, and then define
\[\Tan S = pr_1(pr_2^{-1}(\cT(S))) \subseteq \IP^n.\]
Observe that we can describe $\Tan S$ as well as the variety given by the union of all the (embedded) tangent planes to $S$:
\[\Tan S = \set{x \in \IP^n \st x \in T_P S \text{ for some } P \in S}.\]

\begin{rem}
Let $S$ be a surface embedded in $\IP^n$ by means of the very ample divisor $h$. As a consequence of Remark \ref{rem: geom descr}, we have that if $H$ is $2$-very ample then $T$ is injective.
\end{rem}

\subsection{The variety of tangent planes in the Grassmannian}
Assume from now on that the embedding of $S$ in $\IP^n$ is induced by a $2$-very ample divisor $h$. We want to describe $\cT(S)$ in terms of the $2$-dimensional Schubert cycles of $\IG(2, n) = \Grass(3, n + 1)$.

We give a brief description of the Schubert cycles involved. In codimension $1$ we have the cycle $\sigma_{1, 0, 0}$, representing the set of all the planes in $\IP^n$ intersecting a given $\Lambda_{n - 3}$. The map induced by $\sigma_{1, 0, 0}$ is the Pl\"ucker embedding in $\IP^{N - 1}$ (with $N = \binom{n + 1}{3}$).

In codimension $2$ we have the cycles
\begin{enumerate}
\item $\sigma_{2, 0, 0}$ parametrising all the planes of $\IP^n$ which intersect a given $\Lambda_{n - 4}$;
\item $\sigma_{1, 1, 0}$ parametrising all the planes of $\IP^n$ which intersect a given $\Lambda_{n - 2}$ in (at least) a line.
\end{enumerate}
Finally, in dimension $2$ we have the cycles
\begin{enumerate}
\item $\sigma_{n - 2, n - 2, n - 4}$ parametrising all the planes of $\IP^n$ which are contained in a given $\Lambda_4$ and contain a given line $\Lambda_1 \subseteq \Lambda_4$;
\item $\sigma_{n - 2, n - 3, n - 3}$ parametrising all the planes of $\IP^n$ which are contained in a given $\Lambda_3$ and pass through a given point $P \in \Lambda_3$.
\end{enumerate}

\begin{rem}
Observe that $\sigma_{n - 2, n - 2, n - 4}$ is isomorphic to the Schubert cycle $\sigma_{2, 2, 0} \subseteq \IG(2, 4)$, and that $\sigma_{n - 2, n - 3, n - 3} \simeq \IP^2$.
\end{rem}

The intersection table among these cycles is given in table \ref{tab: intersection g(2, n)}.

\begin{table}[ht]
\begin{tabular}{||c||c|c||}
\hline
\hline
$\cdot$ & $\sigma_{2, 0, 0}$ & $\sigma_{1, 1, 0}$\\
\hline
\hline
$\sigma_{n - 2, n - 2, n - 4}$ & $1$ & $0$\\
\hline
$\sigma_{n - 2, n - 3, n - 3}$ & $0$ & $1$\\
\hline
\hline
\end{tabular}
\caption{Intersection table of the Schubert cycles in $\IG(2, n)$.}
\label{tab: intersection g(2, n)}
\end{table}

We can write
\begin{equation}\label{eq: decomp in g(2, n)}
[\cT(S)] = \alpha' \sigma_{n - 2, n - 2, n - 4} + \beta' \sigma_{n - 2, n - 3, n - 3},
\end{equation}
and we want now to determine the values of $\alpha'$ and $\beta'$.

\subsection{Enumerative meaning of the coefficients}
As we have the description
\begin{equation}\label{eq: t s in grass}
\cT(S) = \set{\pi \in \IG(2, n) \st \pi \text{ is tangent to } S},
\end{equation}
we can give an enumerative meaning to the coefficients $\alpha'$ and $\beta'$ in \eqref{eq: decomp in g(2, n)}.

We start with $\beta' = [\cT(S)] \cdot \sigma_{1, 1, 0}$, which by \eqref{eq: t s in grass} corresponds to the number of tangent planes to $S$ intersecting an $(n - 2)$-dimensional linear subspace in at least one line.

Let $\Lambda_{n - 2} = \Lambda_{n - 1} \cap \Lambda_{n - 1}'$, and call $C = S \cap \Lambda_{n - 1}$ and $C' = S \cap \Lambda_{n - 1}'$. For a generic choice of $\Lambda_{n - 1}$, $\Lambda_{n - 1}'$ we have that $C$ and $C'$ are smooth curves meeting transversely in $d = h^2$ distinct points. Let $P$ be one of these points: then the lines $T_P S \cap \Lambda_{n - 1}$ and $T_P S \cap \Lambda_{n - 1}'$ are distinct and meet only at $P$. So $T_P S \cap \Lambda_{n - 1} \cap \Lambda'_{n - 1} = \{ P \}$, and this means that no line contained in $T_P S$ can be contained in $\Lambda_{n - 2}$. As a consequence, if we have a line contained in $T_P S \cap \Lambda_{n - 2}$, we can assume that $P \in S \smallsetminus \Lambda_{n - 2}$.

\begin{lemma}\label{lemma: sing curve in pencil}
The point $P \in S \smallsetminus \Lambda_{n - 2}$ is such that $T_P S \cap \Lambda_{n - 2}$ is a line if and only if there is a curve $\Gamma$ in the pencil generated by $C$ and $C'$ which is singular at $P$.
\end{lemma}
\begin{proof}
Let $P$ be a point such that $T_P S \cap \Lambda_{n - 2}$ contains a line $l$. Since $P \notin \Lambda_{n - 2}$, there exists a unique hyperplane $\cH$ in the pencil of hyperplanes through $\Lambda_{n - 2}$ passing through $P$. Let $\Gamma = S \cap \cH$, then $\Gamma$ is a curve in the pencil generated by $C$ and $C'$, and since $l \subseteq \cH$ and $P \in \cH$ we deduce that $T_P S \subseteq \cH$ which implies that $\Gamma$ is singular at $P$ (see Remark \ref{rem: jacobian criterion}).

Viceversa, let $\Gamma$ be a curve in the pencil generated by $C$ and $C'$ which is singular at $P$. Then $\Gamma$ is the intersection of $S$ with a hyperplane $\cH$ in the pencil of hyperplanes through $\Lambda_{n - 2}$, and $T_P S \subseteq \cH$ since $\Gamma$ is singular at $P$. But then $T_P S \cap \Lambda_{n - 2} = T_P S \cap \cH \cap \Lambda_{n - 1} = T_P S \cap \Lambda_{n - 1}$ is a line by the Grassmann formula.
\end{proof}

We can then conclude that
\[\beta' = \sum_{C \in \text{Pencil in } |h|} \# (\Sing C),\]
which is a number we can determine.

\begin{lemma}\label{lemma: beta'}
The value of $\beta'$ in \eqref{eq: decomp in g(2, n)} is
\[\beta' = \chitop(S) + h \cdot (2K_S + 3h).\]
\end{lemma}
\begin{proof}
The surface $S$ is embedded in $\IP^n$ by means of the $2$-very ample divisor $h$, hence for a pencil of curves in $|h|$ the generic curve $C$ is smooth of genus $1 + \frac{1}{2} h \cdot (K_S + h)$ and its topological Euler characteristic is
\[\chitop(C) = -h \cdot (K_S + h).\]
The generic singular curve has only one node, and so its topological Euler characteristic is $-h \cdot (K_S + h) + 1$. Finally, the base points of the generic pencil are $d = h^2$ distinct points. We blow up $S$ in these points to find a surface $\tilde{S}$ having a fibration $\tilde{S} \longrightarrow \IP^1$ induced by the pencil. In this setting $\beta'$ corresponds to the number of singular fibres, i.e., to the degree of the discriminant locus of the fibration: we can compute this degree by means of topological methods. Denote by $F$ the generic (smooth) fibre of the fibration, by $\Delta \subseteq \IP^1$ the discriminant locus and by $F_{\Sing}$ the singular fibre over the points of $\Delta$. Choosing the pencil generically, $\Delta$ consists of $\beta'$ distinct points and $F_{\Sing}$ has only one node. Since $\tilde{S}$ is obtained from the surface $S$ after the blow up of $d$ points, we have that $\chitop(\tilde{S}) = \chitop(S) + d$. But then:
\[\begin{array}{rl}
\chitop(S) + d = & \chitop(\tilde{S}) =\\
= & \chitop(\IP^1 \smallsetminus \Delta) \cdot \chitop(F) + \chitop(\Delta) \cdot \chitop(F_{\Sing}) =\\
= & (2 - \beta')(-h \cdot (K_S + h)) + \beta' (-h \cdot (K_S + h) + 1),
\end{array}\]
and this allows us to conclude that $\beta' = \chitop(S) + h^2 + 2 h \cdot (K_S + h)$, which proves the Lemma.
\end{proof}

Now we focus on $\alpha' = [\cT(S)] \cdot \sigma_{2, 0, 0}$. By \eqref{eq: t s in grass} its value corresponds to the number of tangent planes to $S$ meeting a given $(n - 4)$-dimensional linear subspace $\Lambda_{n - 4}$ of $\IP^n$. Such a $\Lambda_{n - 4}$ is the intersection of $4$ hyperplanes, say $\Lambda_{n - 1}$, $\Lambda_{n - 1}'$, $\Lambda_{n - 1}''$ and $\Lambda_{n - 1}'''$: it is not restrictive to assume that each of them cuts $S$ in a smooth curve and that $\Lambda_{n - 4} \cap S = \varnothing$.

\begin{rem}
Under the assumption that $C = S \cap \Lambda_{n - 1}$ is smooth, we have that $\Lambda_{n - 1}$ intersects any tangent space $T_P S$ in a line. In fact, by the Grassmann formula we have $\dim (\Lambda_{n - 1} \cap T_P S) \geq 1$, and finally that $\dim (\Lambda_{n - 1} \cap T_P S) = 2$ if and only if $T_P S \subseteq \Lambda_{n - 1}$ which happens if and only if $C$ is singular at $P$ (see Remark \ref{rem: jacobian criterion}). Moreover, this line is the tangent line to $C$ in $P$.
\end{rem}

\begin{lemma}\label{lemma: from g(2, n) to g(1, n)}
Let $S \subseteq \IP^n$ be a surface embedded by a $3$-very ample divisor $h$. The number $\alpha'$ in \eqref{eq: decomp in g(2, n)} is the degree of $\Tan S$ in $\IP^n$, and it coincides with the number of tangent lines to $S$ which intersect a given $\Lambda_{n - 4} \subseteq \IP^n$.
\end{lemma}
\begin{proof}
For $Q \in S$, we write $l_Q = T_Q S \cap \Lambda_{n - 1}$ (and analogously define $l'_Q$ and so on). Let $\{ P \} = l_Q \cap l'_Q$, we have the following possibilities:
\begin{enumerate}
\item $P = Q$ is one of the $d$ points where $\Lambda_{n - 1} \cap \Lambda_{n - 1}'$ meets $S$.
\item $Q$ is one of the finite number of points which determine a tangent space $T_Q S$ where $l_Q = l'_Q$. It is easy to see that this number is finite since each such $Q$ determines a tangent plane meeting a $\Lambda_{n - 2}$ in at least a line, and we know that there are at most $\beta'$ such planes.
\item $l_Q \neq l'_Q$ and $P \neq Q$. This is the generic situation, and we observe that in this case the point $P$ determines uniquely three data: the tangent plane it belongs to (here is where the $3$-very ampleness come into play, in view of Remark \ref{rem: geom descr}), a tangent line in this tangent plane (the line through $P$ and $Q$), and a length $2$ non-reduced subscheme of $S$ (obtained as the intersection of that tangent line with $S$).
\end{enumerate}

This shows that the intersection of $\Lambda_{n - 1} \cap \Lambda_{n - 1}'$ with $\Tan S$ is a surface, with the property that all but a finite number of its points determine uniquely a tangent plane, a tangent line and a non-reduced length $2$ subscheme. When we intersect again $\Tan S$ with $\Lambda_{n - 1}''$ and $\Lambda_{n - 1}'''$ we get on this surface a finite number of distinct points (this number equals the degree the surface in $\Lambda_{n - 2}$) and so the number of tangent panes to $S$ which intersect a given $\Lambda_{n - 4}$ is the same as the degree of $\Tan S$ as well as the number of tangent lines to $S$ which intersect that given $\Lambda_{n - 4}$.
\end{proof}

\subsection{Reduction to the Grassmannian of lines}
In this Section we assume that the divisor $h$ embedding $S$ in $\IP^n$ is $3$-very ample. Lemma \ref{lemma: from g(2, n) to g(1, n)} allows us to compute the number $\alpha' = [\cT(S)] \cdot \sigma_{2, 0, 0}$ in the Grassmannian $\IG(1, n)$ rather than in $\IG(2, n)$. In fact, the set of lines intersecting a given $\Lambda_{n - 4}$ is described in $\IG(1, n)$ by the codimension $3$ Schubert cycle $\sigma_{3, 0}$, while the set of non-reduced subschemes is the image of the exceptional divisor $E$ of the Hilbert--Chow morphism under the embedding $\varphi: \Hilb^2 S \hookrightarrow \IG(1, n)$. Hence, calling $X$ the image of $E$ in $\IG(1, n)$, we have
\begin{equation}\label{eq: alpha' in other grass}
\alpha' = [X] \cdot \sigma_{3, 0}.
\end{equation}

In the Grassmannian $\IG(1, n)$ we have the codimension $3$ cycles
\begin{enumerate}
\item $\sigma_{3, 0}$ representing all the lines in $\IP^n$ meeting a given $\Lambda_{n - 4}$,
\item $\sigma_{2, 1}$ representing all the lines in $\IP^n$ contained in a $\Lambda_{n - 1}$ and meeting a given $\Lambda_{n - 3} \subseteq \Lambda_{n - 1}$;
\end{enumerate}
and the $3$-dimensional cycles
\begin{enumerate}
\item $\sigma_{n - 1, n - 4}$ representing all the lines in $\IP^n$ contained in a given $\Lambda_4$ and through a fixed point $P \in \Lambda_4$,
\item $\sigma_{n - 2, n - 3}$ representing all the lines in $\IP^n$ contained in a given $\Lambda_3$ and meeting a fixed line $\Lambda_1 \subseteq \Lambda_3$.
\end{enumerate}

We can then write
\begin{equation}\label{eq: decomp x in g(1, n)}
[X] = \alpha' \sigma_{n - 1, n - 4} + \beta'' \sigma_{n - 2, n - 3},
\end{equation}
and we observe that the notation $\alpha'$ in this Section is coherent to the one used in \eqref{eq: decomp in g(2, n)} thanks to Lemma \ref{lemma: from g(2, n) to g(1, n)} and \eqref{eq: alpha' in other grass}.

\begin{lemma}\label{lemma: beta''}
Let $S \subseteq \IP^n$ be a surface embedded by a $3$-very ample divisor $h$, and let $C = S \cap \Lambda_{n - 1}$ be an irreducible smooth hyperplane section of $S$. The coefficient $\beta''$ in \eqref{eq: decomp x in g(1, n)} coincides with the degree of the tangent variety of $C$ in $\Lambda_{n - 1}$, and its value is
\[\beta'' = h \cdot (K_S + 3h).\]
\end{lemma}
\begin{proof}
Observe that $\beta'' = X \cdot \sigma_{2, 1}$ is the number of tangent lines to $S$ which are contained in a given $\Lambda_{n - 1}$ and meet a fixed $\Lambda_{n - 3} \subseteq \Lambda_{n - 1}$. Choosing $\Lambda_{n - 1}$ such that the curve $C = S \cap \Lambda_{n - 1}$ is smooth, it is easy to see that $\beta''$ is the degree of the tangent variety to the curve $C$ in $\Lambda_{n - 1}$: as in the proof of Lemma \ref{lemma: beta}, the smoothness of $C$ implies that a tangent line to $S$ contained in $\Lambda_{n - 1}$ must be a tangent line to $S \cap \Lambda_{n - 1}$.

By \cite[Proposition 3.3]{ac-man-sch}, the degree of the tangent variety of a smooth curve of degree $d$ and genus $g$ in $\IP^n$ is $2d + 2g - 2$: in our case we have $d = h^2$ and $2g - 2 = h \cdot (K_S + h)$, and so we get
\[\beta'' = 2h^2 + h \cdot (K_S + h).\]
\end{proof}

\begin{rem}\label{rem: degree embedding}
We have a linear constraint among $\alpha'$ and $\beta''$: the Pieri formula implies that $\sigma_{1, 0}^3 = \sigma_{3, 0} + 2 \sigma_{2, 1}$, and so we see that $[X] \cdot \sigma_{1, 0}^3 = \alpha' + 2 \beta''$ is the degree of $X$ under the embedding induced by $P = \varphi_{|\sigma_{1, 0}|}$.
\end{rem}

\section{Intersection numbers in the Hilbert scheme}\label{sect: intersection in hilb2}

In this Section we recall some facts about the intersection theory of $\Hilb^2 S$. We mainly focus on the following situation: given an ample divisor $h$ on a surface $S$, we let $H$ be the divisor induced by $h$ on $\Hilb^2 S$ and $E = 2 \delta$ the exceptional divisor of the Hilbert--Chow morphism; we want to compute the intersection numbers $H^4$, $H^3 \delta$, $H^2 \delta^2$, $H \delta^3$ and $\delta^4$.

To fix the notation, we recall and complete the diagram \eqref{eq: second diagram} with the exceptional divisor of the blow up $\eta$:
\[\xymatrix{\widetilde{S \times S} \ar[r]^\eta \ar[d]_\rho & S \times S \ar[d]^\pi\\
\Hilb^2 S \ar[r]^\varepsilon & S^{(2)},} \qquad \qquad \xymatrix{D \ar@{^(->}[r]^j \ar[d]_p & \widetilde{S \times S} \ar[d]^\eta & \\
S \ar@{^(->}[r]^i & S \times S \ar@<0.5ex>[r]^{pr_1} \ar@<-0.5ex>[r]_{pr_2} & S,}\]
where $i: S \longrightarrow S \times S$ is the diagonal embedding and $p: D \longrightarrow S$ is the structure map of the projective bundle $\IP(\cN_{S | S \times S}) \simeq \IP(TS)$ under the isomorphism $D \simeq \IP(\cN_{S | S \times S})$. Given an effective divisor $h$ on $S$ we write $h_i = pr_i^* h$, so $h_1 + h_2$ is invariant under the action exchanging the two factors. As a consequence there exists a divisor $H$ on $\Hilb^2 S$ such that $\rho^* H = \eta^* (h_1 + h_2)$; this $H$ is the divisor induced by $h$ on $\Hilb^2 S$.

The map $\rho$ which was defined as a quotient map on $\widetilde{S \times S}$ can equally be viewed as the order $2$ covering of $\Hilb^2 S$ branched along the exceptional divisor $E$ of the Hilbert--Chow morphism and defined by the line bundle $\cO_{\Hilb^2 S}(\delta)$, where $\delta$ is a divisor such that $E = 2 \delta$. It follows, e.g., from \cite[Lemma I.17.1]{BHPVdV}, that $\rho^* E = 2D$ and $\rho^* \delta = D$. Moreover, by \cite[Example 8.3.12]{intersection-theory} we have that for every pair of cycles $x$ and $y$ on $\Hilb^2 S$ the equality
\[x \cdot y = \frac{1}{2} (\rho^* x \cdot \rho^* y)\]
holds, and this allows us to determine the intersection numbers we are interested in as intersection numbers in $\widetilde{S \times S}$ rather than on $\Hilb^2 S$.

The structure of the Chow ring of $\widetilde{S \times S}$ is related to that of $S \times S$ by \cite[Proposition 6.7]{intersection-theory} for what concerns the additive structure and by \cite[Example 8.3.9]{intersection-theory} for what concerns the multiplicative structure: for every cycle $y, y'$ on $S \times S$ and $x, x'$ on $D$ we have
\[\begin{array}{l}
\eta^* y \cdot \eta^* y' = \eta^*(y \cdot y'),\\
j_* x \cdot j_* x' = j_*(c_1(\cN_{D | \widetilde{S \times S}}) \cdot x \cdot x'),\\
\eta^* y \cdot j_* x = j_*((p^* i^* y) \cdot x).
\end{array}\]

In our situation, we observe the following:
\begin{enumerate}
\item as $i$ is the diagonal embedding, $i^*(h_1 + h_2) = 2h$;
\item as $D \simeq \IP(TS)$ is the projective bundle associated to the tangent bundle $TS$ of $S$, writing $\zeta = c_1(\cO_D(1))$, we have $\cN_{D | \widetilde{S \times S}} = \cO_D(-1)$ and so $c_1(\cN_{D | \widetilde{S \times S}}) = -\zeta$ satisfies the relation
\[\zeta^2 + p^*c_1(TS) \zeta + p^*c_2(TS) = 0.\]
\end{enumerate}

As a consequence of this discussion, we can compute the intersection numbers $H^4$, $H^3 \delta$, $H^2 \delta^2$, $H \delta^3$ and $\delta^4$. Indeed:
\[\begin{array}{rl}
H^4 = & \frac{1}{2} \rho^* H^4 = \frac{1}{2} \eta^*(h_1 + h_2)^4 = \frac{1}{2} 6 h_1^2 h_2^2 =\\
= & 3 h^2 h^2;\\
 & \\
H^3 \delta = & \frac{1}{2} \rho^* H^3 \rho^* \delta = \frac{1}{2} \eta^*(h_1 + h_2)^3 D = \frac{1}{2} \eta^*(h_1 + h_2)^3 j_*(1_D) =\\
= & \frac{1}{2} j_* p^* i^* (h_1 + h_2)^3 = \frac{1}{2} j_* p^* (2h)^3 =\\
= & 0;\\
 & \\
H^2 \delta^2 = & \frac{1}{2} \rho^* H^2 \rho^* \delta^2 = \frac{1}{2} \eta^*(h_1 + h_2)^2 D^2 = \frac{1}{2} \eta^*(h_1 + h_2)^2 \cdot j_*(1_D) j_*(1_D) =\\
= & \frac{1}{2} \eta^*(h_1 + h_2)^2 \cdot j_*(-\zeta) = -\frac{1}{2} j_* (p^* i^*(h_1 + h_2)^2 \cdot \zeta) =\\
= & -\frac{1}{2} j_*(p^*(2h)^2 \cdot \zeta) =\\
= & -2 h^2;\\
 & \\
H \delta^3 = & \frac{1}{2} \eta^* (h_1 + h_2) D \cdot D^2 = \frac{1}{2} j_*(p^*(2h)) j_*(-\zeta) = j_*(p^*h \cdot \zeta^2) =\\
= & -j_*(p^*h \cdot p^* c_1(TS) \cdot \zeta) =\\
= & h \cdot K_S;\\
 & \\
\delta^4 = & \frac{1}{2} (j_* 1_D)^2 \cdot (j_* 1_D)^2 = \frac{1}{2} j_*(\zeta) j_*(\zeta) = -\frac{1}{2} j_*(\zeta^3) =\\
= & \frac{1}{2}(\chitop(S) - K_S^2).
\end{array}\]

The following result is an easy observation, and it gives a condition to transfer the ampleness of a divisor $h$ on $S$ to a divisor on $\Hilb^2 S$ built from the induced divisor $H$.

\begin{lemma}[{\cite[$\S$2]{bertram-coskun}}]\label{lemma: embedding in grass}
Let $S$ be a smooth surface, with a $2$-very ample divisor $h$. Then the divisor $H - \delta$ on $\Hilb^2 S$ is very ample.
\end{lemma}
\begin{proof}
Since $h$ is $2$-very ample we get an embedding
\[\xymatrix{\Hilb^2 S \ar@{^(->}[r]^(0.35){\varphi_1} \ar[dr] & \Grass(2, H^0(S, h)^*) \ar[d]^{\text{Pl\"ucker}}\\
 & \IP^{N - 1},}\]
where the first arrow is the immersion \eqref{eq: hilb to grass map}. By \cite[$\S$2]{bertram-coskun}, the pull-back of the hyperplane divisor of $\IP^{N - 1}$ under the composite embedding is $H - \delta$, which is then a very ample divisor on $\Hilb^2 S$.
\end{proof}

\section{The degree of the secant and tangent variety of a surface}\label{sect: deg sec s, deg tan s}

In this Section we compute the degree of the secant (resp., the tangent) variety of a surface $S$, which is embedded in $\IP^n = \IP(H^0(S, \cO_S(h))^*)$ by means of a $3$-very ample divisor $h$. By Lemma \ref{lemma: alpha} this amounts to compute the value of the coefficient $\alpha$ in \eqref{eq: decomp in g(1, n)}.

\begin{theorem}\label{thm: degree secant variety}
Let $S$ be a surface embedded in $\IP^n = \IP(H^0(S, \cO_S(h))^*)$ by means of a $3$-very ample divisor $h$. The degree of the secant variety to $S$ is
\[\deg \Sec S = \frac{1}{2} (h^2 h^2 - 10 h^2 - 5 h \cdot K_S + \chitop(S) - K_S^2).\]
\end{theorem}
\begin{proof}
To compute the degree of $\Sec S$ we make use of \eqref{eq: degree embedding} together with Lemma \ref{lemma: beta} and Lemma \ref{lemma: gamma}: the embedded variety $\Sigma(S)$ has degree
\[\deg [\Sigma(S)] \cdot \sigma_{1, 0}^4 = \alpha + 3\beta + 2\gamma = \alpha + \frac{5}{2} h^2 h^2 - 7 h^2 - \frac{3}{2} h K_S.\]
On the other hand, by Lemma \ref{lemma: embedding in grass} and the intersection numbers computed in Section \ref{sect: intersection in hilb2},
\[\deg {\sigma_{1, 0}}^4_{|_{\Sigma(S)}} = (H - \delta)^4 = 3 h^2 h^2 - 12 h^2 - 4 h K_S + \frac{1}{2}(\chitop(S) - K_S^2),\]
and so we can now compute the value of $\alpha$.
\end{proof}

As $\sigma_{1, 0}$ is the Schubert cycle representing the lines intersecting a codimension $2$ linear subspace, from the proof of Theorem \ref{thm: degree secant variety} we deduce the following result.

\begin{cor}
Let $S$ be a surface embedded in $\IP^n = \IP(H^0(S, \cO_S(h))^*)$ by means of a $3$-very ample divisor $h$.
\begin{enumerate}
\item There are $3 h^2 h^2 - 12 h^2 - 4 h K_S + \frac{1}{2}(\chitop(S) - K_S^2)$ lines meeting $S$ in $2$ points and four $(n - 2)$-dimensional linear subspaces in general position.
\item The degree of $\Sigma(S)$ under the Pl\"ucker embedding of $\IG(1, n)$ is
\[3 h^2 h^2 - 12 h^2 - 4 h K_S + \frac{1}{2}(\chitop(S) - K_S^2).\]
\end{enumerate}
\end{cor}

We focus now on the tangent variety $\Tan S$.

\begin{theorem}\label{thm: degree tangent variety}
Let $S$ be a surface embedded in $\IP^n = \IP(H^0(S, \cO_S(h))^*)$ by means of a $3$-very ample divisor $h$. The degree of the tangent variety to $S$ is
\[\deg \Tan S = 6 h^2 + 4 h K_S + K_S^2 - \chitop(S).\]
\end{theorem}
\begin{proof}
By Lemma \ref{lemma: from g(2, n) to g(1, n)}, we only need to compute the value of $\alpha'$ in \eqref{eq: decomp x in g(1, n)}. By Remark \ref{rem: degree embedding} and Lemma \ref{lemma: beta''} we have that
\[\alpha' + 2 h K_S + 6 h^2\]
is the degree of the embedded exceptional divisor $E$ of the Hilbert--Chow morphism under the natural inclusion \eqref{eq: hilb to grass map}. As $E = 2\delta$ and this embedding is induced by (the restriction to $E$ of) $\varphi_{|H - \delta|}$, we can compute this degree also as
\[E \cdot (H - \delta)^3 = 2\delta \cdot (H - \delta)^3 = 12 h^2 + 6 h K_S + K_S^2 - \chitop(S).\]
Equating these two expressions leads to $\alpha' = 6 h^2 + 4 h K_S + K_S^2 - \chitop(S)$.
\end{proof}

Combining Theorem \ref{thm: degree tangent variety} with Lemma \ref{lemma: beta'} we have that \eqref{eq: decomp in g(2, n)} reads as
\[\begin{array}{rl}
[\cT(S)] = & (6 h^2 + 4 h K_S + K_S^2 - \chitop(S)) \sigma_{n - 2, n - 2, n - 4} +\\
+ & (\chitop(S) + h \cdot (2K_S + 3h)) \sigma_{n - 2, n - 3, n - 3}.
\end{array}\]
From this expression we easily deduce the following facts.

\begin{cor}
Let $S$ be a surface embedded in $\IP^n = \IP(H^0(S, \cO_S(h))^*)$ by means of a $3$-very ample divisor $h$.
\begin{enumerate}
\item There are $(K_S + 3h)^2$ tangent planes to $S$ which intersect two $(n - 3)$-dimensional linear subspaces in general position.
\item The degree of $\cT(S)$ under the Pl\"ucker embedding of $\IG(2, n)$ is
\[(K_S + 3h)^2.\]
\item The degree of $E$ under the Pl\"ucker embedding of $\IG(1, n)$ (i.e., the degree of the variety parametrising the tangent lines to $S$) is
\[12 h^2 + 6 h K_S + K_S^2 - \chitop(S).\]
\end{enumerate}
\end{cor}
\begin{proof}
Observe that both the number of tangent planes to $S$ which intersect two $(n - 3)$-dimensional linear subspaces in general position and the degree of $\cT(S)$ under the Pl\"ucker embedding of $\IG(2, n)$ can be computed as
\[\begin{array}{rl}
[\cT(S)] \cdot \sigma_{1, 0, 0}^2 = & \cT(S) \cdot (\sigma_{2, 0, 0} + \sigma_{1, 1, 0}) =\\
= & (6 h^2 + 4 h K_S + K_S^2 - \chitop(S)) + (\chitop(S) + h \cdot (2K_S + 3h)) =\\
= & 9h^2 + 6 h K_S + K_S^2 =\\
= & (K_S + 3h)^2.
\end{array}\]
Analogously, the degree of $E$ under the Pl\"ucker embedding of $\IG(1, n)$ is given by
\[[X] \cdot \sigma_{1, 0}^3 = 12 h^2 + 6 h K_S + K_S^2 - \chitop(S).\]
\end{proof}

\section{Two explicit examples}\label{sect: examples}

In this Section we will use Theorem \ref{thm: degree secant variety} and Theorem \ref{thm: degree tangent variety} to compute explicitly the degree of the secant and tangent variety in two concrete examples.

The first one is the one of the $K3$ surfaces, and was the original motivation for the paper (see Remark \ref{rem: motivation}). We focus mainly on the case of the generic projective $K3$ surface, i.e., those having Picard group of rank $1$, because in this case we can link directly the $k$-very ampleness of the generator of the Picard group with its self-intersection.

The second one concerns the surfaces $V_m$ obtained as the image of $\IP^2$ under the $m$-th Veronese embedding.

\subsection{Embedding of \texorpdfstring{$K3$}{K3}'s in \texorpdfstring{$\IP^n$}{Pn}}\label{sect: embedding k3}

Let $S$ be a $K3$ surface with $\Pic S = \IZ \cdot h$, where $h$ is a very ample divisor on $S$ and $h^2 = 2 t$ for some $t \geq 2$. Then we have an embedding $\varphi_{|h|}: S \hookrightarrow \IP^{t + 1} = \IP(H^0(S, \cO_S(h))^*)$. In this Section we study the the $k$-very ampleness of $h$, and in particular we determine the integer $k$ such that $h$ is $k$-very ample but not $(k + 1)$-very ample.

The main tool is the following result by Knutsen, which improves general results of Reider (see \cite{reider} and \cite{beltrametti-francia-sommese}) in the case of $K3$ surfaces.

\begin{theorem}[{\cite[Theorem 1.1]{knutsen}}]\label{thm: k-very ample}
Let $L$ be a nef and big divisor on a $K3$ surface and let $k \geq 0$ be an integer. Then the following are equivalent:
\begin{enumerate}[(i)]
\item $L$ is $k$-very ample;
\item $L^2 \geq 4k$ and there is no effective divisor $D$ satisfying
\begin{equation}\label{eq: cond}
\left\{ \begin{array}{l}
2 D^2 \leq L D \leq D^2 + k + 1 \leq 2k + 2\\
2 D^2 = L D \Longleftrightarrow L \sim 2D \text{ and } L^2 \leq 4k + 4\\
D^2 = k + 1 \Longleftrightarrow L \sim 2D \text{ and } L^2 = 4k + 4.
\end{array} \right.
\end{equation}
\end{enumerate}
\end{theorem}

In our situation it is easy to see that the generator $h$ of $\Pic S$ is very ample for $t \geq 3$ (e.g., by using results in \cite{sdonat}), here we address the question of finding a bound for its $k$-very ampleness. As a direct application of Theorem \ref{thm: k-very ample} we can show that $h$ is $\left[ \frac{t}{2} \right]$-very ample, but not $\left( \left[ \frac{t}{2} \right] + 1 \right)$-very ample. Indeed if $h$ is $k$-very ample, then by Theorem \ref{thm: k-very ample} we conclude that
\[h^2 \geq 4k \leadsto 2t \geq 4k \leadsto k \leq \frac{t}{2}.\]
So $h$ is not $k$-very ample for $k \geq \left[ \frac{t}{2} \right] + 1$. To show that $h$ is $k$-very ample for $k \leq \left[ \frac{t}{2} \right]$ we only have to show that there is no divisor $D = ah$, with $a > 0$, satisfying \eqref{eq: cond}. In particular, we will show that the inequality $2 D^2 \leq 2k + 2$ is never satisfied. In fact, since $2k + 2 \leq t + 2$, from $2 D^2 = 4t a^2 \leq 2k + 2 \leq t + 2$ we deduce that
\[a^2 \leq \frac{t + 2}{4t} = \frac{1}{4} + \frac{1}{2t} \leq \frac{1}{4} + \frac{1}{4} = \frac{1}{2},\]
and so that $a = 0$.

It follows easily from this observation that $h$ is $3$-very ample if $t \geq 6$. As a consequence we can state the following version of Theorem \ref{thm: degree secant variety} and Theorem \ref{thm: degree tangent variety} for the generic projective $K3$ surface. To deduce it, just recall that for a $K3$ surface $S$ we have $K_S = 0$ and $\chitop(S) = 24$.

\begin{prop}\label{prop: deg of sec and tan of k3}
Let $S$ be a $K3$ surface with an ample divisor $h$ such that $\Pic S = \IZ \cdot h$ and $h^2 = 2t$. Use $h$ to embed $S$ in $\IP^{t + 1}$. If $t \geq 6$, then
\begin{enumerate}
\item the degree of the secant variety to $S$ is
\[\deg \Sec S = 2(t - 2)(t - 3);\]
\item the degree of the tangent variety to $S$ is
\[\deg \Tan S = 12(t - 2).\]
\end{enumerate}
\end{prop}

\begin{rem}
If $S$ is a non-generic $K3$ surface, then $\Pic S$ has rank greater than $1$. In this case, if $h$ is a $3$-very ample divisor on $S$ with $h^2 = 2t$, then under the embedding of $S$ in $\IP(H^0(S, \cO_S(h))^*)$ provided by $\varphi_{|h|}$ we can compute the degree of the secant (resp., tangent) variety with the same formulae as in Proposition \ref{prop: deg of sec and tan of k3}. The only difference is that to prove the $3$-very ampleness of $h$, the condition $t \geq 6$ is still necessary but not sufficient, and we need to use the characterization given by Theorem \ref{thm: k-very ample}.
\end{rem}

\begin{rem}\label{rem: motivation}
In the paper \cite{bcns} the authors determine the automorphism group of the Hilbert scheme of two points on a generic projective $K3$ surface, showing that there are at most two automorphisms and giving a characterisation of the cases when there is a non-trivial automorphism. Up to now there is no geometric description of this extra automorphism: the present paper was written while trying to achieve a description, as an automorphism induced by a more natural one defined on the secant variety of the $K3$ surface.
\end{rem}

\subsection{The Veronese embeddings of \texorpdfstring{$\IP^2$}{P2}}\label{sect: veronese}

As a final application, we want to compute the degree of the secant and tangent variety to the (image of the) $m$-th Veronese embedding of $\IP^2$, i.e., of the surface $V_m = \nu_m(\IP^2)$ where
\[\nu_m = \varphi_{|\cO_{\IP^2}(m)|}: \IP^2 \longrightarrow \IP^{\binom{m + 2}{m} - 1}.\]

Before we apply Theorem \ref{thm: degree secant variety} and Theorem \ref{thm: degree tangent variety} to this situation, we need to take a little digression on the expected dimension of the secant variety. The expected dimension of the secant variety to a smooth projective surface is $5$, nevertheless there exist surfaces whose secant variety has lower dimension. Such surfaces are called \emph{defective}, and their classification was a matter of interest for the Italian school of algebraic geometry: we mention Palatini, Scorza, Terracini and Severi who worked on this topic. So we can give a classification for the smooth projective defective surfaces.

\begin{theorem}[{\cite[(5.37), (6.17), (6.18)]{gh-localdifferentialgeometry}}]
Let $S \subseteq \IP^n$ be a surface.
\begin{enumerate}
\item If $S$ has degenerate tangent variety, then either $S$ lies in a $\IP^3$, or else is a cone or a developable ruled surface.
\item If $S$ is defective, then either $S$ is a cone, the tangential ruled surface of a curve in $\IP^4$, or else $S$ lies in a $\IP^5$.
\item In this last case, if moreover the tangent variety is non-degenerate, then either $S$ lies in $\IP^4$ or else $S$ is the Veronese surface $V_2$.
\end{enumerate}
\end{theorem}

\begin{rem}\label{rem: defective surfaces}
If we restrict to the case of smooth projective surfaces embedded in $\IP^n$ with $n \geq 5$, we see easily that the only defective such surface is the Veronese surface $V_2$. The degree of its secant variety is classically known, see, e.g., \cite[p. 179-180]{gh}, and its value is $\deg \Sec V_2 = 3$.
\end{rem}

\begin{rem}
If we consider the Veronese embeddings of $\IP^n$, given by $\nu_{n, m} = \varphi_{|\cO_{\IP^n}(m)|}$, then the problem of determining the dimension of the higher secant varieties $\Sec^k \nu_{n, m}(\IP^n)$ was solved by the Alexander--Hirschowitz Theorem (see \cite{alexander-hirschowitz}): they all have the expected dimension, except for a finite number of cases where the dimension lowers.
\end{rem}

We consider now the variety $V_m$, switching from the language of divisors to the language of line bundles for ease of notation. As it is explained in \cite[Lemma 0.3.5]{beltrametti-sommese-preservationunderadjunction}, the tensor product of $k$ very ample line bundles gives rise to a $k$-very ample line bundle, and as a consequence we have that $\cO_{\IP^2}(m)$ is $3$-very ample for $m \geq 3$ (observe that we are discarding the value $m = 2$ in view of Remark \ref{rem: defective surfaces}). Recalling that $K_{\IP^2} = \cO_{\IP^2}(-3)$ and that $\chitop(\IP^2) = 3$ we easily deduce the following result.

\begin{prop}
Let $V_n \subseteq \IP^{\binom{m + 2}{m} - 1}$ be the image of $\IP^2$ under the $m$-th Veronese embedding $V_m$, with $m \geq 3$. Then
\[\deg \Sec V_m = \frac{1}{2} (m - 1)(m - 2)(m^2 + 3m - 3), \qquad \deg \Tan V_m = 6(m - 1)^2.\]
\end{prop}

\bibliographystyle{plain}
\bibliography{BiblioSecants}

\end{document}